\documentclass[journal]{IEEEtran}
\usepackage{graphics}
\usepackage{graphicx}
\usepackage{times}
\usepackage{amsmath}
\usepackage{amssymb}
\usepackage{flushend}
\usepackage{enumerate}
\usepackage[ruled,vlined]{algorithm2e}
\usepackage{subfigure}
\usepackage{comment}
\usepackage{amsthm}
\usepackage{cite}
\usepackage{xcolor}
\usepackage{epstopdf}
\usepackage{epsfig}
\usepackage{bm}
\usepackage[thicklines]{cancel}

\DeclareMathOperator*{\argmin}{arg\,min}

\title{Distributed Nash Equilibrium Seeking for Constrained Aggregative Games  over Jointly Connected and Weight-Balanced Switching Networks}

\author{Zhaocong~Liu and
	Jie~Huang,~\IEEEmembership{Life Fellow,~IEEE}
			\thanks{This work was supported by the Research Grants Council of the Hong Kong Special Administration Region under grant No. 14203924.}
			\thanks{The authors are with the Department of Mechanical and Automation Engineering, The Chinese University of Hong Kong, Hong Kong
				E-mail: {\tt\small  \{zcliu,jhuang\}@mae.cuhk.edu.hk}.}
		\thanks{Corresponding Author: Jie Huang ({\tt\small  jhuang@mae.cuhk.edu.hk}).}	
}

\newtheorem{theorem}{Theorem}
\newtheorem{proposition}{Proposition}

\newtheorem{lemma}{Lemma}

\newtheorem{remark}{Remark}

\newtheorem{assumption}{Assumption}

\newcommand{\col}{\hbox{col}}
\newcommand{\EQ}{\begin{eqnarray}}
	\newcommand{\EN}{\end{eqnarray}}
\newcommand{\EQQ}{\begin{eqnarray*}}
	\newcommand{\ENN}{\end{eqnarray*}}

\begin{document}
	
	\maketitle

	\begin{abstract}
		The property of the communication network and the constraints on the strategic space are two factors that determine the complexity of the  distributed Nash equilibrium (DNE) seeking problem.
	The DNE seeking problem of aggregative games has been studied  for unconstrained case over all types of communication networks and for various types of constrained games over static and connected communication networks. 
	In this paper, we investigate the DNE seeking problem for constrained aggregative games over jointly connected and weight-balanced switching networks, which can be directed and disconnected at every time instant.  
	By integrating the projected gradient technique and the dynamic average consensus algorithm, we convert our problem to the stability problem of a well-defined time-varying nonlinear system.
	By constructing a time-varying Lyapunov's function candidate for this time-varying nonlinear system, we conduct a rigorous Lyapunov's analysis to conclude the exponential stability 
	of this system and hence solve our problem.  
	\end{abstract}


	\section{Introduction}
	The distributed Nash equilibrium (DNE) seeking problem has been extensively studied in recent years. One of the main challenges of the problem is caused by the fact that the players lack  full information about the actions of all other players. Thus, 
	they have to estimate the actions of other players over a  communication network. Therefore, the nature of the communication network dictates the complexity of the problem. 
	The simplest case occurs when the network is fixed and connected and such a case was studied in, for example, \cite{gadjov2018passivity, ye2017distributed,de2019distributed,feng2023adaptively,huang2024distributed}. 
	References  \cite{ye2017switching,bianchi2020fully,poveda2022fixed} further studied the case where the network is time-varying and every-time connected. The most challenging case is when the players exchange their information over a so-called jointly strongly connected switching network, which can be disconnected at every time instant. 
	Such a case was first studied in \cite{he2023Neurocomput} and was further pursued in several other papers such as \cite{liu2024distributed,liu2024robust}.

	Another challenge of the DNE seeking problem arises from the constraints on the strategic space of the players. 
	Most of the papers cited above studied the unconstrained case, i.e., the strategic space of the players is the whole space.  In practice, due to limitations in the allocated resources or the shortcomings of players' mechanisms, the actions of players may be subject
	to various constraints. One typical constraint is that the players' actions are restricted to certain compact sets. Such a scenario is seen in, for example, the positioning of unmanned aerial vehicles \cite{bianchi2021continuous}, Nash-Cournot games \cite{bianchi2020fully}, distributed resource allocation \cite{gadjov2018passivity},  interference and anti-interference problems \cite{liu2023dynamic}, optical networks \cite{pavel2006noncooperative}, etc. To deal with such a case, one often resorts to the projected gradient-based algorithm. 
	For example,  references \cite{gadjov2018passivity,de2019distributed,huo2024distributed,zou2021continuous} considered this case over fixed and connected communication networks. 
	Reference \cite{he2024distributed} further considered this case over jointly strongly connected switching networks under the assumption that the pseudogradient mapping vanishes at the  unique Nash equilibrium. This assumption was removed recently in \cite{dai2025distributed}.

    There is a special type of games in which the cost function of each player depends on  the player's action and an aggregative function of the actions of all players. Such a game is called aggregative game.
    Unlike the general game where each  player  knows  his/her own cost function,    in an aggregative game, since  
    the cost function of every player $i$ also depends on the aggregate function which is not fully known by the player,  when seeking  the Nash equilibrium of an aggregative game, one not only needs to enable each player to estimate the actions of others, but also 
    enable each player to estimate the  aggregate function. As a result, 
    even though an aggregative game is a special case of a general game, the DNE seeking problem of an aggregative game  presents an additional challenge. 
    The DNE seeking problem in an aggregative game has also been studied by a number of papers over fixed and connected networks in \cite{cai2022distributed,zhu2022asynchronous,deng2018distributed,koshal2016OR}   and 
    time-varying and every-time connected networks in \cite{liang2017distributed}. 
    Very recently, by  combining the  dynamic average consensus protocol and the pseudogradient update module, the DNE seeking problem of aggregative games over a jointly connected and weight-balanced network 
     was further solved in  \cite{liu2024aggregate}. Nevertheless, none of the aforementioned papers considered constrained action space and switching communication networks simultaneously.
     In this paper, we will further consider the constrained  DNE seeking problem of aggregative games over a jointly connected and weight-balanced network.  
    	The main contributions are summarized as follows:
	
	\begin{enumerate}[(1)]
	  	    \item   Compared with \cite{liu2024aggregate}, which considered the unconstrained strategic space, this paper studies the constrained case which cannot be dealt with by the approach of \cite{liu2024aggregate}.
	  	    The difficulty is overcome by applying a projection-based algorithm as will be made clear in Remark \ref{remcontri1}.

            \item  References \cite{liang2022exponentially} and \cite{zhu2020distributed}  studied the  DNE seeking problem of aggregative games over static networks. Their problem  comes down to  the stability issue of a time-invariant system.
            In contrast, our technical challenge is a stability issue of a time-varying system. We need to develop a rigorous Lyapunov approach to conclude the exponential stability of some time-varying nonlinear system. 
            More detailed comparison will be given in   Remark \ref{remcontri2}.
	\end{enumerate}  
	
The new result is obtained by an integration of a projection-based algorithm and the average consensus protocol which is able to estimate the actions of all players and the aggregative function over some jointly connected and weight-balanced network.   
It should be noted that the DNE seeking problem of  aggregative games over  switching networks has also been studied by  
the discrete-time algorithms in \cite{belgioioso2020distributed,liu2024timestamp,liu2023online} and the hybrid dynamic system-based algorithm in \cite{wang2023distributed}.  These approaches are quite different from ours and do not apply to our problem.

	   The rest of the paper is organized as follows.
	   Section \ref{sec:2} provides the preliminaries.
	   Section \ref{sec:3} presents the proposed algorithm and the convergence proof.
	   The conclusion is summarized in Section \ref{sec:5}.

	\indent\textbf{Notations}:
$\mathbb{R}_+$ and $\mathbb{Z}_+$ denote the set of positive numbers and positive integers, respectively.
For vector $x$ or matrix $A$, $\|x\|$ denote the Euclidean norm of $x$ and $\|A\|$ denote the Euclidean-induced matrix norm of $A$.
For column vectors $a_i$,  $\mathrm{col} (a_1,\cdots,a_n)= [a_1^{\rm T},\cdots,a_n^{\rm T}]^{\rm T}$.
$\otimes$ denotes the Kronecker product.
$\bm{1}_p$ is the $p$-dimensional column vector  with all $1$'s, $\bm{0}_{p\times q}$ is the $p\times q$-dimensional matrix with all $0$'s, and $I_p$ is the $p$-dimensional identity matrix.

	\section{Preliminaries}\label{sec:2}
%	 In this section, some basic preliminaries are introduced, 
%	then our problem is formulated.

%A differentiated projection operator of $\mathbf{P}_\Omega(\cdot)$, denoted by $\Pi_\Omega(\cdot,\cdot)$, is defined as
%\begin{align}\label{def 4 Pi(x,v)}
%	\Pi_\Omega(x,v) \overset{\Delta}{=}\lim_{t\to 0_+} \frac{P_\Omega(x+tv)-x}{t}, \forall x\in\Omega, \forall v\in\mathbb{R}^n
%\end{align}
%By \cite[Proposition~III.5.3.5]{hiriart2013convex}, we have
%	\begin{align}\label{equivalent projection}
%	\Pi_\Omega(x,v) = P_{\mathbb{T}_{\Omega}(x)}(v), \hspace{1mm} \forall x\in \Omega, v\in \mathbb{R}^n 
%\end{align}

	\subsection{Game theory}
	A non-cooperative  game is defined by a triplet $\Gamma \overset{\Delta}{=} (\mathcal{V},f_i,U_i)$. Here,  $\mathcal{V}=\{1,\cdots,N\}$ is the set of $N$ players, $U_i\subseteq \mathbb{R}^{n}$ is the action space for player $i$.
	Let $\mathrm{U}=\Pi_{i=1}^NU_i\subseteq\mathbb{R}^{Nn}$ be the strategy space,  $\bm{x}=\mathrm{col}(x_1,\cdots,x_N)\in\mathrm{U}$ be the strategy vector with $x_i\in U_i$ representing player $i$'s action, and  define $\bm{x}_{-i}=\col(x_1,\cdots,x_{i-1},x_{i+1},\cdots,x_N)\in\mathbb{R}^{Nn-n}$. 
	A Nash equilibrium denoted by   $\mathrm{x^*}=\col(\mathrm{x}_i^*,\mathbf{x}_{-i}^*)\in\mathrm{U}$  is such that 
	\begin{align}\label{def 4 NE}
		f_i(\mathrm{x}_i^*,\mathbf{x}_{-i}^*) \leq f_i(\mathrm{x}_i,\mathbf{x}_{-i}^*),~~ \forall i \in \mathcal{V}, \hspace{1mm} \forall \mathrm{x}_i \in U_i.
	\end{align}
	Let $\nabla_if_i(x_i,\bm{x}_{-i})=(\frac{\partial f_i(x_i,\bm{x}_{-i})}{\partial x_i})^{\rm T}\in\mathbb{R}^n$. Then, we call
	\begin{align}\label{psudo-gradient}
		F(\bm{x})= \mathrm{\col} \left( \nabla_1 f_1 (x_1,\bm{x}_{-1}), \cdots, \nabla_N f_N (x_N,\bm{x}_{-N})  \right)
	\end{align}
	the  \emph{pseudo-gradient} operator of the game.

	The following two  assumptions are standard. 
	\begin{assumption}\label{assump: convex fi}
	For all $i \in \mathcal{V}$,  i) $U_i$ is nonempty, closed and convex; ii) the cost function $f_i (x_i, \bm{x}_{-i})$  is convex and continuously differentiable in  ${x}_i$ for every fixed $\bm{x}_{-i} \in U_{-i}$; iii) $F$ in \eqref{psudo-gradient} is $\mu$-strongly monotone on $\mathrm{U}$,
	i.e., for some  $\mu >0$,
	\begin{equation}
				(\bm{x}-\bm{x'})^{\rm T}(F(\bm{x})-F(\bm{x'})) \geq \mu \|\bm{x}-\bm{x'}\|^2 \label{assump: F strong monotone}  
				\end{equation}	
	\end{assumption}

	\begin{assumption}\label{assump: F monotone & theta Lipschitz Globally}
	$F$ in \eqref{psudo-gradient} is   $\theta$-Lipschitz continuous on $\mathrm{U}$, i.e., 
	\begin{align}
	\|F(\bm{x}) - F(\bm{x'})\| \leq \theta\|\bm{x}-\bm{x'}\|, \hspace{2mm} \forall \bm{x},\bm{x'} \in \mathrm{U}   \label{assump: F Lipschitz continuous}	
\end{align}		
\end{assumption}

      \begin{remark}
      By \cite[Prop.~1.4.2]{facchinei2003finite}, under parts (i) and (ii) of Assumption \ref{assump: convex fi},  a pure Nash equilibrium $\mathrm{x^*} \in \mathrm{U}$ exists, which  satisfies the following variational inequality:
      	\begin{align}\label{def 4 variational ineq}
      	(\bm{x}-\mathrm{x^*})^{\rm{T}}F(\mathrm{x^*})\geq 0, ~~ \forall \bm{x}\in\mathrm{U}
      \end{align}
      	
   	Further, by \cite[Thm.~2.3.3]{facchinei2003finite}, under Assumption \ref{assump: convex fi}, a unique NE point $\mathrm{x^*} \in \mathrm{U}$ exists.
     In the special case where $\mathrm{U} =	\mathbb{R}^{Nn}$,  condition \eqref{def 4 variational ineq} reduces to the following 
     \begin{align}\label{def 4 variational eq}
     F(\mathrm{x^*}) = \bm{0}_{(Nn)\times 1}
     \end{align}
 \end{remark}

    To introduce an aggregative game,  let an aggregate function $\sigma(\bm{x})$ be defined as follows:
    \begin{align}\label{def sigma}
    	\sigma(\boldsymbol{x})  \overset{\Delta}{=} \frac{1}{N}\sum_{i=1}^{N} \phi_i(x_i)
    \end{align}
    where $\phi_i(\cdot): \mathbb{R}^{n}\mapsto \mathbb{R}^n$ is a private function known to player $i$. 
    An aggregative game is a game whose cost functions $f_i(x_i,\bm{x}_{-i})$ satisfy $f_i(x_i,\bm{x}_{-i}) = \bar{f}_i(x_i,\sigma(\bm{x}))$ for some functions $\bar{f}_i$. 
    Even though an aggregative game is a special case of a general game, the seeking of the NE over a communication network presents some specific challenge since, 
    unlike the general game where  player $i$  is  aware of his/her own cost function $f_i(x_i,\bm{x}_{-i})$ in \cite{he2023Neurocomput, he2024distributed}  while player $i$ in an aggregative game lacks some information about his/her cost function $\bar{f}_i(x_i, {\sigma}(\bm{x}))$ due to the presence of the unknown functions $\phi_j (j\neq i)$. As a result, in seeking  the Nash equilibrium of an aggregative game over a communication network, one also needs to develop a technique to estimate ${\sigma}(\bm{x})$. 
     To overcome this challenge, we need one more assumption. 
    For this purpose, let $\bm{s} = \col(s_1,s_2,\cdots,s_N)\in \mathbb{R}^{Nn}$ with $s_i\in\mathbb{R}^n$ and $\phi(\bm{x}) \!=\! \col(\phi_1(x_1),\phi_2(x_2),\cdots\!,\phi_N(x_N))\!\in\!\mathbb{R}^{Nn}$.
   We call the following \emph{extended pseudo-gradient} operator
   \begin{align}\label{mathbf F def}
   	\mathbf{F}(\boldsymbol{x},\boldsymbol{s}) = \mathrm{col}(J_1(x_1,s_1),\cdots,J_N(x_N,s_N))
   \end{align}
   where 
   \begin{align}\label{parial gradient Ji Def}
   	J_i(x_i,s_i) &\overset{\Delta}{=} \nabla_y \bar{f}_i(y,s_i)|_{y=x_i} \notag \\
   	&\hspace{5mm} + \frac{1}{N} \nabla\phi_i(x_i)\nabla_y \bar{f}_i(x_i,y)|_{y=s_i}
   \end{align}
   From $f_i(x_i,\bm{x}_{-i}) = \bar{f}_i(x_i,\sigma(\bm{x}))$, and
   Eqs.~\eqref{psudo-gradient}, \eqref{mathbf F def}, \eqref{parial gradient Ji Def}, one has 
   \begin{align}\label{mathbf F def2}
   	\mathbf{F}(\bm{x}, 1_N \otimes\sigma(\bm{x})) = F(\bm{x})
   \end{align}

   \begin{assumption}\label{assump: F 2nd Lipsch & phi bounded}
   	    ~~
    	\begin{enumerate}[1)]
   		\item   $\mathbf{F}$ in \eqref{mathbf F def} is Lipschitz continuous in the second argument, i.e., $\|\mathbf{F}(\bm{x},\bm{s}) - \mathbf{F}(\bm{x},\bm{s'})\|\leq \hat{\theta}\|\bm{s}-\bm{s'}\|, \forall \bm{s},\bm{s'}\in\mathbb{R}^{Nn}$ for some $\hat{\theta}>0$. 
   		
   		\item The Jacobian of $\phi(\bm{x})$ satisfies  $\|\frac{\partial \phi(\bm{x})}{\partial \bm{x}}\|\leq l$ for some $l>0$.
   	\end{enumerate}
   \end{assumption}

     \begin{remark}
     	Part $1)$ of Assumption~\ref{assump: F 2nd Lipsch & phi bounded} is standard and has been used in many literature on aggregative games, see  \cite[Assump.~5]{belgioioso2020distributed}\cite[Assump.~4]{gadjov2020single}\cite[Assump.~3.1)]{liu2024aggregate}\cite[Assump.~3]{zhu2020distributed}\cite[Assump.~3]{zhu2022asynchronous}.
     	Part $2)$ of Assumption~\ref{assump: F 2nd Lipsch & phi bounded} includes the average aggregate function $\phi_i(x_i) = x_i$ in \cite{shakarami2022distributed} and the matrix weighted aggregate function $\phi_i(x_i) = A_ix_i$ in \cite{bianchi2021continuous} as special cases.
     \end{remark}

	\subsection{Graph theory}	
	A time-varying graph is denoted by $\mathcal{G}(t) = (\mathcal{V},\mathcal{E}(t))$, where
	$\mathcal{V} = \{1,\cdots,N\}$ is the node set corresponding to the $N$ players, and $\mathcal{E}(t)\subseteq \mathcal{V}\times \mathcal{V}$ is the edge set. We denote $(j,i)\in \mathcal{E}(t)$ if node $i$ can receive information from node $j$.	
	A directed path from node $i_1$ to node $i_k$ at time $t$ is denoted by $\{(i_1,i_2),\cdots,(i_{k-1},i_k)\}\subseteq \mathcal{E}(t)$.
	The graph $\mathcal{G}(t)$ is said to be \emph{connected} at time $t$ if one node has directed paths to every other node at time $t$, and is said to be \emph{strongly connected} at time $t$ if there is a directed path between any two nodes at time $t$.
	
	Define a piece-wise
	constant switching function $\rho: [0,+\infty) \mapsto \mathcal{P}=\{1,\cdots,n_0\}$ with
	$n_0\in\mathbb{Z}_+$.
	Let $\{t_j:j=0,1,2,\cdots,\}$ be a sequence satisfying $t_0 = 0, t_{j+1}-t_j\geq\tau$ for some constant $\tau>0$ and for all $t\in [t_j,t_{j+1}), \rho(t) = p$ for some $p\in \mathcal{P}$.
	Then, $\mathcal{P}$ is called the switching index set, $t_j$ is called the switching instant, and $\tau$ is called the dwell time.
	
	 Given a set of $n_0$ graphs $\{\mathcal{G}_i = (\mathcal{V}, \mathcal{E}_i), i=1,\cdots,n_0\}$, one can build a time varying graph $\mathcal{G}_{\rho(t)}=(\mathcal{V},\mathcal{E}_{\rho(t)})$ via a piece-wise constant switching signal $\rho(t)$ with range $\mathcal{P}=\{1,\cdots,n_0\}$. We call $\mathcal{G}_{\rho(t)} = (\mathcal{V}, \mathcal{E}_{\sigma(t)})$ a \emph{switching graph} or a \emph{switching network}.	Denote the weighted adjacency matrix of  $\mathcal{G}_{\rho(t)}$ by $\mathcal{A}_{\rho(t)} = [a_{ij}(t)]\in \mathbb{R}^{N\times N}$ where  $a_{ij}(t)>0$ if $(j,i)\in\mathcal{E}_{\rho(t)}$ and $a_{ij}(t)=0$  otherwise.
	 Since there exists no such edge as $(i,i)$, we have $a_{ii}(t)=0$.
	 The in-degree of node $i$ is defined as $d_i^{in} = \sum_{j=1}^N a_{ij}(t)$. Let $D(t) = \mathrm{diag}(d_1^{in},\cdots,d_N^{in})$. The matrix $\mathcal{L}_{\rho(t)} = D(t) - \mathcal{A}_{\rho(t)}$ is called the Laplacian matrix of $\mathcal{G}_{\rho(t)}$.
	 For any $t\geq 0, s>0$, let $\mathcal{G}_{\rho([t,t+s))} = \cup_{t_i\in[t,t+s)} \mathcal{G}_{\rho(t_i)}$.
	 We call $\mathcal{G}_{\rho([t,t+s))}$ the \emph{union graph} of $\mathcal{G}_{\rho(t)}$ over time interval $[t,t+s)$.
	$\mathcal{G}_{\rho(t)}$ is called weight-balanced at time $t$ if $\sum_{j=1}^N a_{ij}(t) = \sum_{j=1}^N a_{ji}(t)$ holds for all $i\in\mathcal{V}$.

	\begin{assumption}\label{assump: jointly connect & weight-balance}
		~~~
		\begin{enumerate}[1)]
			\item  There exists a  positive number $T$ such that the  graph $\mathcal{G}_{\rho([t,t+T))}$ is connected for all $t \geq 0$.
			
			\item  The graph $\mathcal{G}_{\rho(t)}$ is weight-balanced for any $t\geq 0$.
		\end{enumerate}    	
	\end{assumption}
	        \begin{remark}
	        	For convenience, we say a switching graph $\mathcal{G}_{\rho(t)}$ satisfying Assumption~\ref{assump: jointly connect & weight-balance} is jointly connected and weight-balanced. Under Assumption~\ref{assump: jointly connect & weight-balance}, $\mathcal{G}_{\rho(t)}$ can be disconnected for any time, thus is the mildest one in existing literature of distributed Nash equilibrium seeking for aggregative games~\cite{liu2024aggregate}.
			\end{remark}

	\section{Main Result}\label{sec:3}
	
	For a game satisfying  Assumptions~\ref{assump: convex fi} and \ref{assump: F monotone & theta Lipschitz Globally},  by viewing the action  variables  to be governed by the following first-order integrator dynamics
	\begin{align}\label{def 4 single integra}
		\dot{x}_i(t) = u_i(t), \hspace{1mm} i\in\mathcal{V}
	\end{align}
we can treat our problem  to that of finding a distributed control protocol $u_i$ such that the solution of the closed-loop system converges to the Nash equilibrium. 
	The  special case of our problem where $\mathrm{U} =	\mathbb{R}^{Nn}$ was studied in \cite{liu2024aggregate}.  But, when  $\mathrm{U}$ is a compact set, the approach in \cite{liu2024aggregate} does not work since 
	the control protocol in \cite{liu2024aggregate} cannot guarantee the action variables $x_i (t)$ belongs to  $\mathrm{U}_i$ even if $x_i (0) \in \mathrm{U}_i$.
The standard way for dealing with this difficulty is to introduce the projection operator. 
	Let $\Omega\!\subset\!\mathbb{R}^n$ be a closed convex set.
	%	Two types of projections will be used in this paper.
	For a vector $x\in\mathbb{R}^n$, the (Euclidean) projection operator $\mathbf{P}_{\Omega}(\cdot)$ is defined as $\mathbf{P}_{\Omega}(x)\overset{\Delta}{=}\argmin_{x'\in\Omega}\|x'-x\|^2$, which is to find a unique vector $x'\in\Omega$ that is closest to $x$ in the Euclidean norm. By \cite[Theorem~1.5.5 (d)]{facchinei2003finite}, $\mathbf{P}_{\Omega}(\cdot)$ is non-expansive, i.e., 
	\begin{align}\label{nonexpansive of Eucli project}
		\|\mathbf{P}_\Omega(x)-\mathbf{P}_\Omega(y)\|\leq\|x-y\|, \hspace{2mm} \forall x,y\in\mathbb{R}^n
	\end{align}	
	%The tangent cone of $\Omega$ at $x\in\Omega$ is a closed convex cone which is defined as 
	%\begin{align}\label{def 4 tangent cone}
	%	\mathbb{T}_\Omega(x) = \overline{\cup_{\epsilon>0} \frac{1}{\epsilon}(\Omega-x)}
%	\end{align}

   We now define our control protocol  for each player $i$ as follows
	\begin{subequations}\label{ctrllaw 4 aggregate}
		\begin{align}
			u_i &= \delta_2\big(\mathbf{P}_{U_i}(x_i-\delta_1J_i(x_i, s_i))-x_i\big) \label{ctrllaw 4 aggregate xi} \\
			\dot{s}_i &= -\alpha(s_i-\phi_i(x_i)) - \beta\sum_{j=1}^N a_{ij}(t)(s_i-s_j) - v_i \label{ctrllaw 4 aggregate si}\\
			\dot{v}_i &= \alpha\beta\sum_{j=1}^N  a_{ij}(t)(s_i-s_j), \hspace{2mm} \sum_{j=1}^N v_j(0)=\bm{0}_{n\times 1}  \label{ctrllaw 4 aggregate vi}
		\end{align}
	\end{subequations}  
	where $s_i\in\mathbb{R}^n,v_i\in\mathbb{R}^n$ are two internal variables, $\delta_1, \delta_2, \alpha, \beta$ are adjustable constant parameters to be specified later. The zero sum initial condition indicated in \eqref{ctrllaw 4 aggregate vi} will be utilized in  Proposition~\ref{proposition: alternate sys} later.

	Let $\bm{x}=\col(x_1,\cdots\!,x_N)\!\!\in\!\!\mathbb{R}^{Nn}, \bm{s}=\col(s_1,\cdots\!,s_N)\!\!\in\!\!\mathbb{R}^{Nn}, \bm{v}=\col(v_1,\cdots\!,v_N)\!\!\in\!\!\mathbb{R}^{Nn}$. Then, the compact form of Eqs. \eqref{def 4 single integra} and \eqref{ctrllaw 4 aggregate xi}-\eqref{ctrllaw 4 aggregate vi} is as follows
	\begin{subequations}\label{compact 4 aggregate}
		\begin{align}
			\dot{\bm{x}} &= \delta_2\big(\mathbf{P}_{\mathrm{U}}(\bm{x}-\delta_1\mathbf{F}(\bm{x},\bm{s}))-\bm{x}\big)   \label{compact 4 aggregate x} \\
		\dot{\bm{s}} &= -\alpha(\bm{s}-\phi(\bm{x})) - \beta\mathbf{L}_{\rho(t)}\bm{s} - \bm{v}   \label{compact 4 aggregate s} \\
		\dot{\bm{v}} &= \alpha\beta\mathbf{L}_{\rho(t)}\bm{s}, \hspace{2mm} \sum_{j=1}^N v_j(0)=\bm{0}_{n\times 1} \label{compact 4 aggregate v}
		\end{align}
	\end{subequations}
	where $\mathbf{L}_{\rho(t)} := \mathcal{L}_{\rho(t)}\otimes I_n$.

	 	\begin{remark} \label{remcontri1}
	 	As explained in \cite{liu2024aggregate}, $J_i(x_i,s_i)$ is used to estimate  the pseudo-gradient $F$ defined in \eqref{psudo-gradient}, $s_i \in \mathbb{R}^n$ is used to estimate  the unknown aggregate function $\sigma(\bm{x})$, and $v_i \in \mathbb{R}^n$ is employed to compensate the mismatch between the $N$ local functions $\phi_i(x_i)$ and the aggregate function $\sigma(\bm{x})$ at steady state.	
	 	The R.H.S. of  \eqref{compact 4 aggregate x}	is modified from \cite[Eq.~(11a)]{dai2025distributed},  which studied the DNE seeking for general games.  The non-expansive condition \eqref{nonexpansive of Eucli project} together with Assumptions~\ref{assump: F monotone & theta Lipschitz Globally} and \ref{assump: F 2nd Lipsch & phi bounded} guarantee that  the R.H.S. of \eqref{compact 4 aggregate} is Lipschitz in its all arguments, thus for any initial condition, the solution of \eqref{compact 4 aggregate} exists and is unique.  Moreover, 
	 	since $\mathbf{P}_{\mathrm{U}}(\bm{x}-\delta_1\mathbf{F}(\bm{x},\bm{s}))\in\mathrm{U}$, $\dot{\bm{x}}\in\mathbb{T}_{\mathrm{U}}(\bm{x})$ for any $\bm{x}\in\mathrm{U}$ where 
	 	$ \mathbb{T}_{\mathrm{U}}(\bm{x})$ is the  tangent cone of $\mathrm{U}$ at $\bm{x}$,  
	 	by Nagumo's theorem in \cite[pp.~174 \& 214]{aubin1984differential},  $\bm{x}(t)\in \mathrm{U}, \forall t\geq 0$ if  $x_i(0)\in U_i$ for all $i\in\mathcal{V}$ in \eqref{compact 4 aggregate x}.
	 	If we replace \eqref{ctrllaw 4 aggregate xi} by $u_i = -\delta_1  J_i(x_i, s_i)$, then \eqref{ctrllaw 4 aggregate} reduces to the one used in \cite{liu2024aggregate}. However, in this case, 
	 	even if   $x_i(0)\in U_i$ for all $i\in\mathcal{V}$, there is no guarantee that $\bm{x}(t)\in \mathrm{U}, \forall t\geq 0$. Thus, the control law in \cite{liu2024aggregate} does not apply to the constrained case. 

	 \end{remark}

By  \cite[Prop.~1.5.8]{facchinei2003finite} or \cite[Lem.~2.38]{ruszczynski2011nonlinear}, $\mathrm{x^*}$ is an NE if and only if
\begin{align}\label{NE x* equality}
	\mathrm{x^*} = \mathbf{P}_\mathrm{U}(\mathrm{x^*}-kF(\mathrm{x^*})), \hspace{2mm} \forall k>0
\end{align}
Using \eqref{NE x* equality} shows that $(\mathrm{x^*}, 1_N\otimes\sigma(\mathrm{x^*}), \alpha(I_{Nn}-\frac{\bm{1}_N\bm{1}_N^{\rm T}}{N}\otimes I_n)\phi(\mathrm{x^*}) )$ is an equilibrium of \eqref{compact 4 aggregate}. Thus, if 
the solution of \eqref{compact 4 aggregate} converges to this equilibrium, then the NE is obtained.

	To facilitate subsequent convergence analysis of \eqref{compact 4 aggregate},  like in \cite{liu2024aggregate},
	let $Q = [r, R]\in\mathbb{R}^{N\times N}$ be an orthogonal matrix with $r = \frac{\bm{1}_N}{\sqrt{N}}\in\mathbb{R}^N$ and  
	$R\in\mathbb{R}^{N\times (N-1)}$. 
	One can verify that $r^TR=\bm{0}_{1\times(N-1)}$ and $R^TR=I_{N-1}$.
	Also, let  $\mathrm{P}_n = \mathbf{rr}^{\rm T} = \frac{\bm{1}_N \bm{1}_N^{\rm T}}{N} \otimes I_n, \mathrm{P}_{n}^\perp =I_{Nn}-\mathbf{rr}^{\rm T}=I_{Nn} - \frac{\bm{1}_N \bm{1}_N^{\rm T}}{N} \otimes I_{n}$, and
	\begin{align}\label{shorthand rR otimes}
		\mathbf{Q}  = Q\otimes I_n = [r\otimes I_n, R\otimes I_n]=[\mathbf{r}, \mathbf{R}]
	\end{align}
	The matrices $\mathrm{P}_n$ and $\mathrm{P}_{n}^\perp$  represent the operations of projecting a vector onto the $n$-dimensional consensus and dispersion spaces, respectively.

	Performing the following coordinate transformation on  \eqref{compact 4 aggregate s} and \eqref{compact 4 aggregate v}
		\begin{subequations}\label{error Def 4 bar-s-v}
		\begin{align}
			\bar{\bm{s}} &= \bm{s}- \mathrm{P}_n\phi(\bm{x})   \label{error Def 4 bar-s-v 1} \\
			\bar{\bm{v}} &=  \bm{v}- \alpha \mathrm{P}_n^\perp\phi(\bm{x}) \label{error Def 4 bar-s-v 2}
		\end{align}
	\end{subequations}
 and utilizing the identity $\mathbf{L}_{\rho(t)}(\bm{1}_N \otimes \sigma(\bm{x}))=(\mathcal{L}_{\rho(t)}\bm{1}_N)\otimes\sigma(\bm{x})=\bm{0}$ for all $t\geq 0$ gives  the following system
	\begin{subequations}\label{error system 4 bar-s-v}
		\begin{align}
			\dot{\bar{\bm s}} &= -\alpha\bar{\bm{s}} - \beta\mathbf{L}_{\rho(t)}\bar{\bm{s}} - \bar{\bm{v}}  - \mathrm{P}_n\frac{\partial \phi(\bm{x})}{\partial \bm{x}}\dot{\bm{x}}  \label{error system 4 bar-s-v 1} \\
			%%%%%%%%%%%%%%%%%%%%%%%%%%%%%%%%%%%%%%%%%%%%%%%%%%%%%%%%%%%%%%%%%%%%%%%%%%%%%%
			\dot{\bar{\bm{v}}} &= \alpha\beta\mathbf{L}_{\rho(t)}\bar{\bm{s}}  - \alpha \mathrm{P}_n^\perp\frac{\partial \phi(\bm{x})}{\partial \bm{x}}\dot{\bm{x}}  \label{error system 4 bar-s-v 2}
		\end{align}
	\end{subequations}
Further, let   
	\begin{subequations}\label{error Def 4 es-ev}
		\begin{align}
			\bm{e}_s &=  \mathbf{Q}^{\rm T}\bar{\bm{s}}  = \begin{bmatrix}
				\mathbf{r}^{\rm T}\\ \mathbf{R}^{\rm T}
			\end{bmatrix}\bar{\bm{s}} = \begin{bmatrix}
				\bm{e}_{s1}\\\bm{e}_{s2}
			\end{bmatrix} \label{error Def 4 es-ev 1} \\
			%%%%%%%%%%%%%%%%%%%%%%%%%%%%%%%%%%%%%%%%
			\bm{e}_v &=    \mathbf{Q}^{\rm T}\bar{\bm{v}}  = \begin{bmatrix}
				\mathbf{r}^{\rm T}\\ \mathbf{R}^{\rm T}
			\end{bmatrix}\bar{\bm{v}} = \begin{bmatrix}
				\bm{e}_{v1}\\\bm{e}_{v2}
			\end{bmatrix}   \label{error Def 4 es-ev 2}
		\end{align}
	\end{subequations}  
	where  $\bm{e}_{s1},\bm{e}_{v1}\in\mathbb{R}^{n}$ and $\bm{e}_{s2},\bm{e}_{v2}\in\mathbb{R}^{Nn-n}$.
	By \eqref{error Def 4 es-ev} and the weight-balanced condition in Assumption~\ref{assump: jointly connect & weight-balance}, we can put Eqs.~\eqref{error system 4 bar-s-v} into the following  form
	\begin{subequations}\label{error system 4 es-ev}
		\begin{align}
			\dot{\bm{e}}_{s1} &= -\alpha\bm{e}_{s1} - \bm{e}_{v1} - \mathbf{r}^{\rm T}\frac{\partial \phi(\bm{x})}{\partial \bm{x}}\dot{\bm{x}}  \label{error system 4 es1} \\
			%%%%%%%%%%%%%%%%%%%%%%%%%%%%%%%%%%%%%%%%
			\dot{\bm{e}}_{s2} &= -\alpha\bm{e}_{s2} - \beta\mathbf{R}\!^{\rm T}\mathbf{L}_{\rho(t)}\mathbf{R}\bm{e}_{s2} - \bm{e}_{v2}   \label{error system 4 es2} \\
			%%%%%%%%%%%%%%%%%%%%%%%%%%%%%%%%%%%%%%%%
			\dot{\bm{e}}_{v1} &= \bm{0}_{n\times 1}  \label{error system 4 ev1} \\
			%%%%%%%%%%%%%%%%%%%%%%%%%%%%%%%%%%%%%%%%
			\dot{\bm{e}}_{v2} &= \alpha\beta\mathbf{R}\!^{\rm T}\mathbf{L}_{\rho(t)}\mathbf{R}\bm{e}_{s2} - \alpha \mathbf{R}^{\rm T}\frac{\partial \phi(\bm{x})}{\partial \bm{x}}\dot{\bm{x}} \label{error system 4 ev2}
		\end{align}
	\end{subequations}
	where we have utilized identities $\mathbf{r}^{\rm T}\mathbf{L}_{\rho(t)}=\frac{1}{\sqrt{N}}(\bm{1}_N^{\rm T}\mathcal{L}_{\rho(t)})\otimes I_n=\bm{0}_{n\times(Nn)}, \mathbf{r}^{\rm T}\mathrm{P}_n=\mathbf{r}^{\rm T}, \mathbf{r}^{\rm T}\mathrm{P}^\perp_n=\bm{0}_{n\times(Nn)}, \mathbf{R}^{\rm T}\mathrm{P}_n=\bm{0}_{(Nn-n)\times Nn}, \mathbf{R}^{\rm T}\mathrm{P}_n^\perp=\mathbf{R}^{\rm T}$.

The following proposition simplifies our problem to the exponential stability of system  \eqref{alternate compact 4 aggregate}.

	\begin{proposition}\label{proposition: alternate sys}
	Consider the following system:
		\begin{subequations}\label{alternate compact 4 aggregate}
			\begin{align}
				\dot{\bar{\bm{x}}} &= g_0(\bm{x},\bm{s})   \label{alternate compact 4 aggregate x}  \\
				%%%%%%%%%%%%%%%%%%%%%%%%%%%%%%%%%%%%%%%%
				\dot{\bm{e}}_{s1} &= -\alpha\bm{e}_{s1}  - \mathbf{r}^{\rm T}\frac{\partial \phi(\bm{x})}{\partial \bm{x}}g_0(\bm{x},\bm{s})  \label{alternate compact 4 aggregate es1} \\
				%%%%%%%%%%%%%%%%%%%%%%%%%%%%%%%%%%%%%%%%
				\dot{\bm{e}}_{s2} &= -\alpha\bm{e}_{s2} - \beta\mathbf{R}\!^{\rm T}\mathbf{L}_{\rho(t)}\mathbf{R}\bm{e}_{s2} - \bm{e}_{v2}   \label{alternate compact 4 aggregate es2} \\
				%%%%%%%%%%%%%%%%%%%%%%%%%%%%%%%%%%%%%%%%
				\dot{\bm{e}}_{v2} &= \alpha\beta\mathbf{R}\!^{\rm T}\mathbf{L}_{\rho(t)}\mathbf{R}\bm{e}_{s2} - \alpha \mathbf{R}^{\rm T}\frac{\partial \phi(\bm{x})}{\partial \bm{x}}g_0(\bm{x},\bm{s})  \label{alternate compact 4 aggregate ev2}
			\end{align}
		\end{subequations}
		   where $\bar{\bm{x}} = \bm{x} - \mathrm{x^*}$ and 
		   \begin{align}\label{def 4 g0}
		   	    g_0(\bm{x},\bm{s}) = \delta_2\big(\mathbf{P}_{\mathrm{U}}(\bm{x}-\delta_1\mathbf{F}(\bm{x},\bm{s}))-\bm{x}\big)
		   \end{align}
	\end{proposition}
		Under part $2)$ of Assumptions \ref{assump: jointly connect & weight-balance},   if  system \eqref{alternate compact 4 aggregate} is exponentially stable with its domain of attraction containing any initial  $x_i(0)\in U_i$, any $\bm{e}_{s} (0)$ and any $\bm{e}_{v2} 	(0)$,	 then, for any initial condition $x_i(0)\in U_i$, any  	$s_i(0)\in \mathbb{R}^n$, and $v_i(0) \in \mathbb{R}^n$ such that $\sum_{i=1}^N v_i(0)=\bm{0}_{n\times 1}$,
 the solution of \eqref{compact 4 aggregate} exponentially converges to the following 
\begin{subequations}\label{prop 4 respective final convergence}
	\begin{align}
			&\lim_{t\to +\infty} \bm{x}(t) = \mathrm{x^*} \\
			&\lim\limits_{t\to +\infty} \bm{s}(t) =  \mathrm{P}_n\phi(\mathrm{x^*}) \\
		&\lim\limits_{t\to +\infty} \bm{v}(t) =  \alpha\mathrm{P}_n^\perp\phi(\mathrm{x^*})
	\end{align}
\end{subequations}

		\begin{proof}
Since $\lim\limits_{t\to +\infty} \!\bar{\bm{x}}(t)\! = \bm{0}_{(Nn)\times 1}$,  exponentially,  
we have 
	\begin{align}
		\lim_{t\to +\infty} \bm{x}(t) = \lim\limits_{t\to +\infty} (\bar{\bm{x}}(t) + \mathrm{x^*}) = \mathrm{x^*}
	\end{align}
Combining \eqref{error Def 4 bar-s-v 2}, \eqref{error Def 4 es-ev 2} and the initial condition $\sum_{j=1}^N v_j(0)=\bm{0}_{n\times 1}$ in \eqref{compact 4 aggregate v}, we have $\bm{e}_{v1}(0)=\mathbf{r}^{\rm T}(\bm{v}(0)-\alpha\mathrm{P}_n^\perp\phi(\bm{x}(0)))=\mathbf{r}^{\rm T}\bm{v}(0)=\bm{0}_{n\times 1}$. Thus,  $\bm{e}_{v1}(t)=\bm{e}_{v1}(0)=\bm{0}_{n\times 1}$ for all $t\geq 0$ by \eqref{error system 4 ev1}, which together with the fact that $\lim\limits_{t\to +\infty} \bm{e}_s(t) = \bm{0}_{(Nn)\times 1}, \lim\limits_{t\to +\infty} \bm{e}_{v2}(t) = \bm{0}_{(Nn-n)\times 1}$ both 
exponentially implies the following:   
	\begin{subequations}\label{main thm1 4 respective final converg s and nu}
		\begin{align}
			&\lim\limits_{t\to +\infty} \!\bm{s}(t) = \lim\limits_{t\to +\infty} (\mathbf{Q}\bm{e}_s(t) \!+\! \mathrm{P}_n\phi(\bm{x}(t)\!)) = \mathrm{P}_n\phi(\mathrm{x^*}) \\
			&\lim\limits_{t\to +\infty} \!\bm{v}(t) = \lim\limits_{t\to +\infty} \!\big(\mathbf{Q}\bm{e}_v(t) \!+\! \alpha\mathrm{P}_n^\perp\phi(\bm{x}(t)\!)\big) \notag \\
			&\hspace{15.5mm} = \alpha\mathrm{P}_n^\perp\phi(\mathrm{x^*})
		\end{align}
	\end{subequations}
	both exponentially. 
	The proof is thus complete.

	\end{proof}

	   Let  $\bm{\zeta}=\col(\bm{e}_{s1},\bm{e}_{s2},\bm{e}_{v2})\in\mathbb{R}^{2Nn-n}$ and define 
	    \begin{align}\label{def 4 switching At}
	    	A(t) \!=\! \begin{bmatrix}
	    		-\alpha I_n \!\!\!&\!\!\! \bm{0} \!\!\!&\!\!\! \bm{0}\\
	    		\bm{0} \!&\! -\alpha I_{Nn-n}\!-\!\beta\mathbf{R}^{\rm T}\mathbf{L}_{\rho(t)}\mathbf{R} \!&\! -I_{Nn-n} \\
	    		\bm{0} \!&\! \alpha\beta\mathbf{R}^{\rm T}\mathbf{L}_{\rho(t)}\mathbf{R} \!&\! \bm{0}
	    	\end{bmatrix}
	    \end{align}
	    Then, system~ \eqref{alternate compact 4 aggregate} can be put in the following compact form: 
	    \begin{subequations}\label{final compact 4 aggregate}
	    	\begin{align}
	    		\dot{\bar{\bm{x}}} &= g_0(\bm{x},\bm{s})   \label{final compact 4 aggregate x}  \\
	    		%%%%%%%%%%%%%%%%%%%%%%%%%%%%%%%%%%%%%%%%
	    		\dot{\bm{\zeta}} &= A(t)\bm{\zeta} - \begin{bmatrix}
	    			\mathbf{r}^{\rm T}\frac{\partial \phi(\bm{x})}{\partial \bm{x}}g_0(\bm{x},\bm{s}) \\ \bm{0}_{(Nn-n)\times 1} \\ \alpha\mathbf{R}^{\rm T}\frac{\partial \phi(\bm{x})}{\partial \bm{x}}g_0(\bm{x},\bm{s})
	    		\end{bmatrix}  \label{final compact 4 aggregate zeta} 
	    	\end{align}
	    \end{subequations}

	    The following  lemma originally established in \cite[Lem.~1]{liu2024aggregate} lays the foundation of our main result. 
        \begin{lemma}\label{lem: exponential stability}
        	Under part $1)$ of Assumption~\ref{assump: jointly connect & weight-balance}, the origin of the linear switched system \begin{equation}\label{unforced sys}
        		\dot{\check{\zeta}}=A(t)\check{\zeta}
        	\end{equation} is exponentially stable.
        \end{lemma}

	   \begin{remark}\label{rem: properties of H(t)}
	   	   By the proof of \cite[Thm.~1]{liu2024aggregate}, Lemma~\ref{lem: exponential stability} ascertains the existence of a time-varying bounded matrix $H(t)\in\mathbb{R}^{(2Nn-n)\times(2Nn-n)}$ such that
	   	   $h_1\|v\|^2\leq v^{\rm T}H(t)v\leq h_2\|v\|^2$ for two positive constants $h_1,h_2>0$ and any vector $v$,  and,  on each time interval $[t_j,t_{j+1})$ with $j=0,1,\cdots,$  $H(t)$ satisfies the following differential Lyapunov equation
	   	   \begin{align}\label{differential Lyap eqn2}
	   	   	-\dot{H}(t) &= A^{\rm T}(t)H(t) + H(t)A(t) + I_{2Nn-n}
	   	   \end{align}
	   	   Since $H (t)$ is bounded,  there exists a positive constant $p>0$ such that 
	   	   \begin{align}\label{bounded H(t)}
	   	   	\|H(t)\|\leq p, \hspace{3mm} \forall t\geq 0
	   	   \end{align}

	   	   Note that \eqref{final compact 4 aggregate zeta} can be viewed as a perturbed version of the unforced system \eqref{unforced sys}.
	   	   Therefore, the existence of $H(t)$ provides a natural construction of the Lyapunov function candidate \eqref{overall Lyapunov v1 + v2} in Theorem~\ref{thm: aggregate case} later.
	   \end{remark}

	Then we present the following main theorem. 
	\begin{theorem}\label{thm: aggregate case}
		Under Assumptions~\ref{assump: convex fi} to  \ref{assump: jointly connect & weight-balance},  let $\delta_1^*=\frac{2\mu}{\theta^2}>0$ and 
		$\delta_2^*(\delta_1) = \frac{k_1(\delta_1)}{k_1(\delta_1)k_3(\delta_1)+k_2^2(\delta_1)}>0$, where 
		\begin{subequations}\label{params of delta_123 M}
			\begin{align}
				k_1(\delta_1) &= \frac{\delta_1(2\mu - \delta_1\theta^2)}{2+\delta_1\theta} \label{params of delta_1}  \\
				k_2(\delta_1) &= \frac{(\delta_1\theta+2)M+\delta_1\hat{\theta}}{2} \label{params of delta_2} \\
				k_3(\delta_1) &= \delta_1M\hat{\theta} \label{params of delta_3}  \\
				M &= 2pl\sqrt{\alpha^2+1}  \label{params of M}
			\end{align}
		\end{subequations}	
		Then, 	for any $\alpha>0, \beta>0$, 
		any  $0<\delta_1<\delta_1^*$,  $0<\delta_2<\delta_2^*(\delta_1)$,  any initial conditions $x_i(0)\in U_i, s_i(0)\in \mathbb{R}^n$, and $v_i(0)\in \mathbb{R}^n$ satisfying $\sum_{i=1}^{N} v_i(0) = \mathbf{0}_{n\times 1}$, the solution $\mathrm{col}(\bm{x}(t),\bm{s}(t),\bm{v}(t))$ of the closed-loop system \eqref{compact 4 aggregate} converges to the equilibrium $\mathrm{col}(\mathrm{x^*},\mathrm{P}_n\phi(\mathrm{x^*}),\alpha\mathrm{P}^\perp_n\phi(\mathrm{x^*}))$ exponentially.
		\end{theorem}
	\begin{proof}	
	 By Proposition~\ref{proposition: alternate sys},  under part $2)$ of Assumption~\ref{assump: jointly connect & weight-balance}, it suffices to show that the solution of 
	system \eqref{alternate compact 4 aggregate}, or equivalently,  \eqref{final compact 4 aggregate} converges exponentially to its equilibrium point for any 
	$x_i(0)\in U_i$, any $\bm{e}_{s} (0)$ and $\bm{e}_{v2}(0)$. 
	To this end, first note that by Remark \ref{remcontri1},   $x_i(t)\in U_i$ for all $t \geq 0$. 
	Let $V_1(\bar{\bm{x}}) = \frac{1}{2}\|\bar{\bm{x}}\|^2$,  $ V_2(\bm{\zeta},t) =  \bm{\zeta}^{\rm T}H(t)\bm{\zeta}$. Consider a time-varying Lyapunov function for the closed-loop system \eqref{final compact 4 aggregate} as follows
		\begin{align}\label{overall Lyapunov v1 + v2}
			\quad V(\bm{\bar{x}},\bm{\zeta},t) 
			= V_1(\bar{\bm{x}}) + V_2(\bm{\zeta},t) = \frac{1}{2}\|\bar{\bm{x}}\|^2 + \bm{\zeta}^{\rm T}H(t)\bm{\zeta}  
		\end{align}
	Then,  by Remark~\ref{rem: properties of H(t)}, 
	 $V(\bm{\bar{x}},\bm{\zeta},t)$ is positive-definite, proper, and decrescent  in the sense that $\min\{\frac{1}{2},h_1\}(\|\bar{\bm{x}}\|^2\!+\!\|\bm{\zeta}\|^2)\leq V(\bm{\bar{x}},\bm{\zeta},t)\leq \max\{\frac{1}{2},h_2\}(\|\bar{\bm{x}}\|^2\!+\!\|\bm{\zeta}\|^2)$.

		The time derivative of $V_1$ along the solution of \eqref{final compact 4 aggregate x} satisfies 
		\begin{align}\label{V1 dot}
			\dot{V}_1  &= \delta_2\bar{\bm{x}}^{\rm T}\big(\mathbf{P}_{\mathrm{U}}(\bm{x}-\delta_1\mathbf{F}(\bm{x},\bm{s}))-\bm{x}\big) \nonumber \\	%%%%%%%%%%%%%%%%%%%%%%%%%%%%%%%%%%%%%%%%%%%%%%%%%%%%%%%%%%%%%%%%%%%%%%%%
			&\overset{(a)}{=} \delta_2\bar{\bm{x}}^{\rm T}\big(\mathbf{P}_{\mathrm{U}}(\bm{x}-\delta_1\mathbf{F}(\bm{x},\bm{s})\!)-\mathbf{P}_{\mathrm{U}}(\bm{x}-\delta_1\mathbf{F}(\bm{x},\mathrm{P}_n\phi(\bm{x})\!)\!)\big)   \notag \\
			 &\hspace{5mm}  +\delta_2\bar{\bm{x}}^{\rm T}\big(\mathbf{P}_{\mathrm{U}}(\bm{x}-\delta_1\mathbf{F}(\bm{x},\mathrm{P}_n\phi(\bm{x})\!)\!)-\bm{x}\big) \notag \\
			%%%%%%%%%%%%%%%%%%%%%%%%%%%%%%%%%%%%%%%%%%%%%%%%%%%%%%%%%%%%%%%%%%%%%%%
			&\overset{(b)}{\leq} \delta_1\delta_2\|\bar{\bm{x}}\|\hat{\theta}\|\bm{s}\!-\! \mathrm{P}_n\phi(\bm{x})\| \!+  \delta_2\bar{\bm{x}}^{\rm T}\big(\mathbf{P}_{\mathrm{U}}(\bm{x}\!-\!\delta_1\mathbf{F}(\bm{x},\mathrm{P}_n\phi(\bm{x})\!)\!) \notag \\
			&\hspace{5mm} -\mathbf{P}_{\mathrm{U}}(\mathrm{x^*}-\delta_1F(\mathrm{x^*})\!)\big) -\delta_2\bar{\bm{x}}^{\rm T}\big(\bm{x}-\mathrm{x^*}\big) \notag \\
			%%%%%%%%%%%%%%%%%%%%%%%%%%%%%%%%%%%%%%%%%%%%%%%%%%%%%%%%%%%%%%%%%%%%%%%%		
			&\overset{(c)}{\leq} \delta_1\delta_2\hat{\theta}\|\bar{\bm{x}}\|\|\bar{\bm{s}}\|  \!+\!  \delta_2\|\bar{\bm{x}}\|\|\bar{\bm{x}}-\!\delta_1(F(\bm{x})\!-\!F(\mathrm{x^*})\!)\| \!-\! \delta_2\|\bar{\bm{x}}\|^2  
		\end{align}
		   where $(a)$ results from adding and subtracting the term $\mathbf{P}_{\mathrm{U}}(\bm{x}-\delta_1\mathbf{F}(\bm{x},\mathrm{P}_n\phi(\bm{x})\!)\!)$, $(b)$ follows from \eqref{nonexpansive of Eucli project},  \eqref{NE x* equality}, and part $1)$ of Assumption~\ref{assump: F 2nd Lipsch & phi bounded}, $(c)$ is derived from \eqref{nonexpansive of Eucli project}, \eqref{error Def 4 bar-s-v 1}, and the identity \eqref{mathbf F def2}. 

        For the last two terms in \eqref{V1 dot}, like in \cite{dai2025distributed}, we consider the following two cases:
        \begin{enumerate}[(i)]
        	\item When $\bar{\bm{x}}\neq\bm{0}$, one has
        	\begin{align}\label{merge two terms}
        		&\hspace{5mm} \delta_2\|\bar{\bm{x}}\|\|\bar{\bm{x}}-\!\delta_1(F(\bm{x})\!-\!F(\mathrm{x^*})\!)\| \!-\! \delta_2\|\bar{\bm{x}}\|^2 \notag \\
        		&= -\delta_2\|\bar{\bm{x}}\|\frac{\|\bar{\bm{x}}\|^2-\|\bar{\bm{x}}-\!\delta_1(F(\bm{x})\!-\!F(\mathrm{x^*})\!)\|^2}{\|\bar{\bm{x}}\|+\|\bar{\bm{x}}-\!\delta_1(F(\bm{x})\!-\!F(\mathrm{x^*})\!)\|}  \notag \\
        		&= -\delta_2\|\bar{\bm{x}}\|\frac{2\delta_1\bar{\bm{x}}^{\rm T}(F(\bm{x})\!-\!F(\mathrm{x^*})\!)\!-\delta_1^2\|F(\bm{x})\!-\!F(\mathrm{x^*})\|^2}{\|\bar{\bm{x}}\|+\|\bar{\bm{x}}-\!\delta_1(F(\bm{x})\!-\!F(\mathrm{x^*})\!)\|}  \notag \\
        		&\overset{(a)}{\leq} -\delta_2\|\bar{\bm{x}}\|\frac{2\delta_1\mu\|\bar{\bm{x}}\|^2 - \delta_1^2\theta^2\|\bar{\bm{x}}\|^2}{\|\bar{\bm{x}}\|+\|\bar{\bm{x}}-\!\delta_1(F(\bm{x})\!-\!F(\mathrm{x^*})\!)\|} \notag \\
        		&\overset{(b)}{\leq} -\delta_2\|\bar{\bm{x}}\|\frac{2\delta_1\mu\|\bar{\bm{x}}\|^2 - \delta_1^2\theta^2\|\bar{\bm{x}}\|^2}{2\|\bar{\bm{x}}\|+\delta_1\theta\|\bar{\bm{x}}\|} \notag \\
        		&=  -\delta_1\delta_2\frac{2\mu - \delta_1\theta^2}{2+\delta_1\theta}\|\bar{\bm{x}}\|^2
        	\end{align}
        	    To derive the numerator in $(a)$, we have made use of the strong monotone property and Lipschitz continuity of $F(\cdot)$ by part $iii)$ of Assumption~\ref{assump: convex fi} and Assumption~\ref{assump: F monotone & theta Lipschitz Globally}, which yields
        	\begin{align}
        		&\hspace{5mm}2\delta_1\bar{\bm{x}}^{\rm T}(F(\bm{x}) \!-\!F(\mathrm{x^*})\!)-\delta_1^2\|F(\bm{x})-F(\mathrm{x^*})\|^2 \notag \\
        		&\geq 2\delta_1\mu\|\bar{\bm{x}}\|^2 - \delta_1^2\theta^2\|\bar{\bm{x}}\|^2 > 0
        	\end{align}
        	    The last strict positive sign follows from $0<\delta_1<\delta_1^*=\frac{2\mu}{\theta^2}$.
        	    To derive~$(b)$, we have utilized the $\theta$-Lipschitz continuity of $F(\cdot)$ by Assumption~\ref{assump: F monotone & theta Lipschitz Globally}.
        	
        	\item When $\bar{\bm{x}}=\bm{0}$, i.e., $\bm{x}=\mathrm{x^*}$, inequality~\eqref{merge two terms} obviously holds since both sides are zero.       	
        \end{enumerate}
       
             Using \eqref{merge two terms} in  \eqref{V1 dot} gives
             \begin{align}\label{V1 dot simplified}
             	\dot{V}_1 &\leq \delta_1\delta_2\hat{\theta}\|\bar{\bm{x}}\|\|\bar{\bm{s}}\|  -\frac{\delta_1\delta_2(2\mu - \delta_1\theta^2)}{2+\delta_1\theta}\|\bar{\bm{x}}\|^2 \notag \\
             	&\leq  \delta_1\delta_2\hat{\theta}\|\bar{\bm{x}}\|\|\bm{\zeta}\|  -\frac{\delta_1\delta_2(2\mu - \delta_1\theta^2)}{2+\delta_1\theta}\|\bar{\bm{x}}\|^2
             \end{align}
                where the last inequality follows from $\|\bar{\bm{s}}\|=\|\mathbf{Q}\bm{e}_s\|=\|\bm{e}_s\|\leq\|\col(\bm{e}_s,\bm{e}_{v2})\|=\|\bm{\zeta}\|$ by \eqref{error Def 4 es-ev 1}.

		%		Let us define  $M_1=2p\sqrt{\alpha^2+1}>0$.	
		Next, consider the time derivative of $V_2$ w.r.t. \eqref{final compact 4 aggregate zeta}.
		For any $t\in[t_j,t_{j+1})$ with $j=0,1,\cdots,$ one has
		\begin{align}\label{V2 dot}
			\dot{V}_2 &= \bm{\zeta}^{\rm T}H(t)\dot{\bm{\zeta}} + \bm{\zeta}^{\rm T}\dot{H}(t)\bm{\zeta}  + \dot{\bm{\zeta}}^{\rm T}H(t)\bm{\zeta} \notag \\
			%%%%%%%%%%%%%%%%%%%%%%%%%%%%%%%%%%%%%%%%%%%%%%%%%%%%%%%%%%%%%%%%%%%%%%%
			&= \bm{\zeta}^{\rm T}(A^{\rm T}(t)H(t) + \dot{H}(t) + H(t)A(t))\bm{\zeta}  \notag \\
			&\hspace{4mm} - 2\bm{\zeta}^{\rm T}H(t)\begin{bmatrix}
				\mathbf{r}^{\rm T}\frac{\partial \phi(\bm{x})}{\partial \bm{x}}g_0(\bm{x},\bm{s}) \\ \bm{0}_{(Nn-n)\times 1} \\ \alpha\mathbf{R}^{\rm T}\frac{\partial \phi(\bm{x})}{\partial \bm{x}}g_0(\bm{x},\bm{s})
			\end{bmatrix}  \notag \\
			%%%%%%%%%%%%%%%%%%%%%%%%%%%%%%%%%%%%%%%%%%%%%%%%%%%%%%%%%%%%%%%%%%%%%%%
			&\overset{(a)}{\leq}  -\|\bm{\zeta}\|^2 \!+ 2\sqrt{\!\alpha^2\!+\!1}\|\bm{\zeta}\|\!\|H(t)\|\!\|\mathbf{Q}^{\rm T}\|\!\|\frac{\partial \phi(\bm{x})}{\partial \bm{x}}\|\!\|g_0(\bm{x},\bm{s})\|     \notag \\
			%%%%%%%%%%%%%%%%%%%%%%%%%%%%%%%%%%%%%%%%%%%%%%%%%%%%%%%%%%%%%%%%%%%%%%%
			&\overset{(b)}{\leq}  -\|\bm{\zeta}\|^2 \!+ 2pl\sqrt{\!\alpha^2\!+\!1}\|\bm{\zeta}\|\|g_0(\bm{x},\bm{s})\| 
		\end{align}
		    where $(a)$ follows from  \eqref{differential Lyap eqn2} and the Cauchy-Schwarz inequality, $(b)$ is derived from \eqref{bounded H(t)}, $\|\mathbf{Q}^{\rm T}\|=1$, and part $2)$ of Assumption~\ref{assump: F 2nd Lipsch & phi bounded}.
		 
		   Before we proceed, let us give an estimate of the norm of $g_0$  as follows
		   \begin{align}\label{norm of g0}
		   	   &\hspace{5mm}\|g_0(\bm{x},\bm{s})\| \notag \\
		   	   &= \delta_2\|\mathbf{P}_{\mathrm{U}}(\bm{x}-\delta_1\mathbf{F}(\bm{x},\bm{s})\!)-\bm{x}\| \notag \\
		   	   &\overset{(a)}{\leq} \delta_2\|\mathbf{P}_{\mathrm{U}}(\bm{x}-\delta_1\mathbf{F}(\bm{x},\bm{s})\!)-\mathbf{P}_{\mathrm{U}}(\bm{x}-\delta_1\mathbf{F}(\bm{x},\mathrm{P}_n\phi(\bm{x})\!)\!)\| \hspace{.5mm}   \notag \\
		   	   &\hspace{5mm} +\delta_2\|\mathbf{P}_{\mathrm{U}}(\bm{x}\!-\!\delta_1\mathbf{F}(\bm{x},\mathbf{P}_n\phi(\bm{x})\!)\!)\!-\!\mathbf{P}_\mathrm{U}(\mathrm{x^*}\!-\!\delta_1F(\mathrm{x^*})\!)\| \notag \\
		   	   &\hspace{5mm} + \delta_2\|\bm{x}-\mathrm{x^*}\|  \notag \\
		   	   &\overset{(b)}{\leq} \delta_1\delta_2\|\mathbf{F}(\bm{x},\bm{s})-\mathbf{F}(\bm{x},\mathrm{P}_n\phi(\bm{x})\!)\| +  \delta_1\delta_2\|F(\bm{x})-F(\mathrm{x^*})\| \notag \\
		   	   &\hspace{5mm} + 2\delta_2\|\bar{\bm{x}}\| \notag \\
		   	   &\overset{(c)}{\leq} \delta_1\delta_2\hat{\theta}\|\bar{\bm{s}}\| +  \delta_1\delta_2\theta\|\bar{\bm{x}}\| +  2\delta_2\|\bar{\bm{x}}\| \notag \\
		   	    &\leq  \delta_1\delta_2\hat{\theta}\|\bm{\zeta}\| +  (\delta_1\theta+2)\delta_2\|\bar{\bm{x}}\| 
		   \end{align}
		    where we have used \eqref{NE x* equality} to derive $(a)$ since $\delta_1>0$,  used \eqref{nonexpansive of Eucli project} and the identity $\mathbf{F}(\bm{x},\mathrm{P}_n\phi(\bm{x}))\!=\!F(\bm{x})$ to derive $(b)$, and  used Assumption~\ref{assump: F monotone & theta Lipschitz Globally}, part $1)$ of Assumption~\ref{assump: F 2nd Lipsch & phi bounded} to derive $(c)$.

		    Define $M=2pl\sqrt{\alpha^2+1}>0$.
		    Then, by \eqref{norm of g0}, one can further simplify \eqref{V2 dot} as follows
		    \begin{align}\label{V2 dot simplified}
		    	\dot{V}_2 
		    	&\leq  -(1-\delta_1\delta_2M\hat{\theta})\|\bm{\zeta}\|^2 + (\delta_1\theta+2)\delta_2M\|\bm{\zeta}\|\|\bar{\bm{x}}\| 
		    \end{align}
		
		 Combining \eqref{overall Lyapunov v1 + v2}, \eqref{V1 dot simplified} and \eqref{V2 dot simplified} gives
		   	\begin{align}\label{final Vdot}
		   	&\dot{V} =  \dot{V}_1 + \dot{V}_2 \notag \\
		   	\leq& -   \begin{bmatrix}
		   		\|\bar{\bm{x}}\| \\ \|\bm{\zeta}\|
		   	\end{bmatrix}^{\rm T}\underbrace{\begin{bmatrix}
		   		\delta_2	k_1(\delta_1) & - \delta_2 k_2(\delta_1) \\
		   			- \delta_2 k_2(\delta_1) & 1-\delta_2k_3(\delta_1)
		   	\end{bmatrix}}_{B(\delta_1,\delta_2)}\begin{bmatrix}
		   		\|\bar{\bm{x}}\| \\ \|\bm{\zeta}\|
		   	\end{bmatrix}
		   \end{align}
		   where $k_1(\delta_1), k_2(\delta_1), k_3(\delta_1)$ are defined in \eqref{params of delta_1}-\eqref{params of delta_3}, respectively.
		   For any positive $\delta_1$ such that $\delta_1<\delta_1^*=\frac{2\mu}{\theta^2}$, one has $k_1(\delta_1)>0$ by \eqref{params of delta_1}.
		   Then, $\delta_2k_1(\delta_1)>0$ for any $\delta_2>0$. Setting the determinant of $B(\delta_1,\delta_2)$ to be greater than zero gives 
		   \begin{align}\label{determinant of B_delta_12}
		   	   \text{det}\big(B(\delta_1,\delta_2)\big) = \delta_2(k_1-(k_1k_3+k_2^2)\delta_2) > 0
		   \end{align}
		   One can  verify that,  for any $0<\delta_2<\delta_2^*=\frac{k_1(\delta_1)}{k_1(\delta_1)k_3(\delta_1)+k_2^2(\delta_1)}$, inequality~\eqref{determinant of B_delta_12} holds. In this case,   \eqref{final Vdot} yields
		   \begin{align}\label{simplified final Vdot}
		   	    \dot{V} &\leq -\lambda_{min}(B(\delta_1,\delta_2))(\|\bar{\bm{x}}\|^2+\|\bm{\zeta}\|^2) \notag \\
		   	    &\leq -\frac{\lambda_{min}(B(\delta_1,\delta_2))}{\max\{\frac{1}{2},h_2\}}V \\
		   	   \Longrightarrow \hspace{2mm} & V(t)\leq V(0)e^{-\frac{\lambda_{min}(B(\delta_1,\delta_2))}{\max\{\frac{1}{2},h_2\}}t}
		   \end{align}  
		  That is,    
		  system \eqref{alternate compact 4 aggregate} or \eqref{final compact 4 aggregate} is exponentially stable with its domain of attraction containing any initial conditions $x_i(0)\in U_i$, any $\bm{e}_{s} (0)$ and any $\bm{e}_{v2}(0)$. The proof is thus complete by noting Proposition~\ref{proposition: alternate sys}. 
		   	\end{proof}

	\begin{remark} \label{remGlobal}
	Since the solution of  system \eqref{compact 4 aggregate}  exists globally, from the proof of Theorem \ref{thm: aggregate case}, it is not difficult to see that 
	system \eqref{compact 4 aggregate}  is globally exponentially stable if
	Assumptions~\ref{assump: convex fi} and  \ref{assump: F monotone & theta Lipschitz Globally} are strengthened so that  inequalities 	\eqref{assump: F strong monotone}  and 
		\eqref{assump: F Lipschitz continuous}
		hold for all $\bm{x},\bm{x'} \in \mathbb{R}^{Nn}$. In this case, Theorem \ref{thm: aggregate case} holds for all $x_i(0) \in \mathbb{R}^n$.
		\end{remark}

	   \begin{remark} \label{remcontri2}
	   	References~\cite{liang2022exponentially} and \cite{zhu2020distributed} studied  distributed NE seeking for constrained aggregative games with local set constraints.
	   	Compared with \cite{liang2022exponentially,zhu2020distributed}, our work offers at least three new features as follows:
	   	\begin{enumerate}[1)]
	   		\item Communication network: The approaches proposed in \cite{zhu2020distributed,liang2022exponentially} only apply to fixed and strongly connected graph while  our approach  apply to jointly connected and weight-balanced switching networks which can be directed and disconnected for any time.   
	   		\item Convergence speed: \cite[Thm.~1 \& 2]{zhu2020distributed} only ensure asymptotical stability of both algorithms at the NE point. In contrast, our result guarantees exponential convergence.
	   		\item Parameter adjustment: Using our terminologies, the last line of \cite[Eqs.~(17), (18)]{zhu2020distributed} can be put as follows
	   		\begin{align}\label{decreasing alpha(t)}
	   			\dot{\bm{x}} &= \alpha(t)\big(\mathbf{P}_{\mathrm{U}}(\bm{x}-\mathbf{F}(\bm{x},\bm{s}))-\bm{x}\big)   
	   		\end{align}
	   		where $\alpha(t)>0$ is a decreasing function satisfying $\int_{0}^{\infty} \alpha(t)dt=\infty$ and $\int_{0}^{\infty} \alpha^2(t)dt<\infty$.	         
	   		In contrast, our projected gradient-play module in~\eqref{compact 4 aggregate x} used fixed gains $\delta_1,\delta_2$ to fine tune the algorithm, which  increases the convergence speed and simplify the complexity of overall Lyapunov stability analysis.   
	   	\end{enumerate}
	   \end{remark}

	   \begin{remark}\label{rem: compare with Dai}
	   	  Reference \cite{dai2025distributed} studied the NE seeking for general games on compact sets over jointly strongly connected switching networks. 
	   	  However, the problem in \cite{dai2025distributed} is quite different from the problem here because,  as pointed out in Section II,  the  player $i$ here   lacks some information about its cost function $\bar{f}_i(x_i, {\sigma}(\bm{x}))$ due to the presence of the unknown functions $\phi_j$ for $j\neq i$. 
	   	   Thus,  the overall distributed dynamics~\eqref{compact 4 aggregate} is totally different from \cite[Eq.~(11)]{dai2025distributed} since, as in \cite{he2023Neurocomput, he2024distributed},  
	   	   each player $i$ of  \cite{dai2025distributed}  only needs to estimate  all  players' actions by a distributed estimator~\cite[Eq.~(11b)]{dai2025distributed}. In contrast,  we not only need 
	   	   to estimate  all  players' actions but also the unknown aggregate function, which 
	   	   cannot be done by the distributed estimator in \cite{he2023Neurocomput, he2024distributed, dai2025distributed} and has to be done by the dynamic average consensus module~\eqref{compact 4 aggregate s}-\eqref{compact 4 aggregate v}.  As a result, the convergence analysis of \eqref{compact 4 aggregate} is much more complicated than the convergence analysis of \cite[Eq.~(11)]{dai2025distributed}.  
	   	   Nevertheless,   since,  for $N$ player games with $x_i\in\mathbb{R}^n$, \cite[Eq.~(11b)]{dai2025distributed} requires each player to exchange a $Nn$ dimensional vector with his/her neighbors, the total dimension for \cite[Eq.~(11)]{dai2025distributed} is $Nn+N^2n$. In contrast, to implement the estimation module~\eqref{compact 4 aggregate s}-\eqref{compact 4 aggregate v}, each player only needs to exchange a $2n$ dimensional vector with others, which means that the total dimension of~\eqref{compact 4 aggregate} is $3Nn$, which is strictly less than $Nn+N^2n$ for $N>2$.
	    Thus,  the communication and computation burden of the algorithm here  is much smaller than the one in \cite{dai2025distributed}.  
	   \end{remark}

	\section{Conclusion}\label{sec:5}
	In this paper, we have studied the DNE seeking problem for constrained aggregative games over jointly connected and weight-balanced switching networks, which can be directed and disconnected for every time instant.  
	By integrating the projected gradient technique and the dynamic average consensus algorithm, we have converted our problem to the stability problem of a time-varying nonlinear system which is solved by establishing its exponential stability.  
	Although we have not yet considered coupling constraints, a natural extension for future work is to investigate the generalized Nash equilibrium seeking problem.

	%	\bibliographystyle{plain}
	%	\bibliography{reference}

	%\begin{IEEEbiography}[{\includegraphics[width=1in,height=1.25in,clip,keepaspectratio]{LZC}}]{Zhaocong Liu}
	%    received the B.Eng degree in mechatronic engineering from Zhejiang University, Hangzhou, China, in 2018, the
	%   M.S. degree in control science and engineering from the
	%   Shanghai Jiao Tong University, Shanghai, China, in 2021.
	%   He is currently pursuing the Ph.D. degree with the Department
	%   of Mechanical and Automation Engineering, The Chinese University of Hong Kong, Hong Kong, SAR, China.
	%
	%   His research interests include the game theory, safety critical control, discrete event systems and nonlinear output regulation.
	%\end{IEEEbiography}
	%
	%
	%\begin{IEEEbiography}[{\includegraphics[width=1in,height=1.25in,clip,keepaspectratio]{Prof huang.png}}]{Jie Huang} (Life Fellow, IEEE) received the Diploma
	%	degree from Fuzhou University, Fuzhou, China,
	%	the master’s degree from Nanjing University
	%	of Science and Technology, Nanjing, China, and
	%	the Ph.D. degree from Johns Hopkins University,
	%	Baltimore, MD, USA.
	%	
	%	He is a Choh-Ming Li Research Professor with
	%	the Shenzhen Research Institute and the Department
	%	of Mechanical and Automation Engineering, The
	%	Chinese University of Hong Kong, Hong Kong
	%	SAR, China. His research interests include nonlinear control theory and applications, multi-agent systems, and flight guidance and control.
	%	
	%	Dr. Huang is a Fellow of IFAC, CAA, and HKIE.
	%\end{IEEEbiography}
	
\end{document}